\documentclass[12 pt]{amsart}

\usepackage{amsthm}
\usepackage{amsmath}
\usepackage{amssymb}
\usepackage{latexsym}
\usepackage{url}
\usepackage{graphicx}
\usepackage{bmpsize}

\author{Felix Goldberg}
\address{Caesarea-Rothschild Institute, University of Haifa, Haifa, Israel}
\email{felix.goldberg@gmail.com}

\title[Lowers bounds for the principal eigenvector of a graph]{A lower bound on the entries of the principal eigenvector of a graph}
\date{March 6, 2014}

\newtheorem{thm}{Theorem}[]

\newtheorem{lem}[thm]{Lemma}

\newtheorem{assume}[thm]{Assumption}

\theoremstyle{example}

\newtheoremstyle{example_contd}
{\topsep} {\topsep}%
{\upshape}
{}
{\bfseries\scshape}
{.}
{1em}
{\thmname{#1} \thmnumber{ #2}\thmnote{#3} (continued)}

\theoremstyle{example_contd}

\begin{document}

\begin{abstract}
We obtain a lower bound on each entry of the principal eigenvector of a non-regular connected graph.
\end{abstract}



\thanks{{This research was supported by the Israel Science Foundation (grant number 862/10.)}}

\maketitle

\section{Introduction}
The theory of graph spectra, in whose earliest annals we find such illustrious names as Hoffman, Bose, Seidel, and Fiedler, has by now attained a fairly mature stage. Recent expositions of the theory may be found in the books \cite{Spectra_BH,Spectra_CRS,AGT2,Spectra_Mieghem}.

On the other hand, while the theory of graph eigenvectors may be already out of its infancy, it is still very much in a state of toddlerhood. The purpose of the present note is to make a modest contribution to one of the basic problems of this theory - the description of the entries of the principal eigenvector of a non-regular graph.

\section{The problem}
Let $G$ be a connected graph on $n$ vertices with adjacency matrix $A \in \mathbb{R}^{n \times n}$. The following facts are widely known (and may be found in each of the references mentioned above):

\begin{itemize}
\item
$A$ is an irreducible nonnegative matrix.
\item
The spectral radius $\rho(G)$ of $A$ is a simple eigenvalue.
\item
The eigenvector $x \in \mathbb{R}^{n}$ corresponding to $\rho$ is positive entrywise.
\end{itemize}

We shall refer to $\rho(G)$ as the spectral radius of $G$ and when the context is clear, denote simply $\rho=\rho(G)$. The vector $x$ will be referred to as the \emph{principal eigenvector} of $G$. An alternative name, which we shall not use here, would be the \emph{Perron vector}. 

It is also very well known that:
\begin{itemize}
\item
If $G$ is regular, then all the entries of $x$ are equal.
\end{itemize}

Papendieck and Recht in \cite{PapRech00} were the first to study the problem of estimating the entries of $x$ in the case that $G$ is non-regular. Before presenting their result, we make an assumption which will be sustained throughout the rest of the note:
\begin{assume}
The vector $x$ is normalized so that $\sum_{i=1}^{n}{x_{i}^{2}}=1$.
\end{assume}

\begin{thm}\cite{PapRech00}\label{thm:pap}
Let $G$ be a connected graph with principal eigenvector $x$. Let $x_{\max}$ be the largest entry of $x$. Then
$$
x_{\max} \leq \frac{1}{\sqrt{2}}.
$$
Equality is attained if and only if $G=K_{1,n-1}$ is the star on $n$ vertices.
\end{thm}

In fact, Papendieck and Recht's full result is more general, holding for every $p$-norm ($p \in [1,\infty]$) and depending also on $\rho$. However, in the case of interest to us, $p=2$, it reduces to $\frac{1}{\sqrt{2}}$.

\section{Known bounds on $x$}

Let us introduce some more notation: denote the degree of the $i$th vertex of $G$ by $d_{i}$. The subgraph of $G$ obtained by deleting the $i$th vertex (and all edges incident on it) will be denoted as $G_{(i)}$. The spectral radius of $G_{(i)}$ will be denoted by $\rho_{i}$. Note that since $G$ is connected, we have by \cite[Corollary 2.1.5(b)]{Avi}:
$$
\rho>\rho_{i}.
$$

Cioab\u{a} and Gregory \cite{CioGre07pr} have generalized Theorem \ref{thm:pap} to give upper bounds on every entry of $x$. 

\begin{thm}\cite[Theorem 3.2]{CioGre07pr}\label{thm:cg}
Let $G$ be a connected graph with principal eigenvector $x$. Then for every $1 \leq i \leq n$:
$$
x_{i} \leq \frac{1}{\sqrt{1+\frac{\rho^{2}}{d_{i}}}}.
$$
Equality is attained if and only if $x_{i}=x_{\max}$, $d_{i}=n-1$, and $G_{(i)}$ is regular.
\end{thm}

A natural counterpart to Theorem \ref{thm:cg} is given by Li, Wang, and Van Mieghem \cite{LiWangMie12}:
\begin{thm}\cite{LiWangMie12}\label{thm:mieghem}
Let $G$ be a connected graph with principal eigenvector $x$. Then for every $1 \leq i \leq n$:
$$
x_{i} \geq \sqrt{\frac{\rho-\rho_{i}}{2\rho}}.
$$
\end{thm}

We remark that additional bounds for $x_{\max}$ and $x_{\min}$ can be found in \cite{CioGre07pr,Nik08}. There are also in the literature results of a different kind where $\sum_{i \in S}{x_{i}^{2}}$ is estimated from above for subsets $S \subseteq V(G)$ which induce either empty \cite{Cio10,LiWangMie12} or, more generally, regular subgraphs \cite{AndCar13}. When $S$ is a singleton set such bounds reduce to an analogue of Theorem \ref{thm:cg}.

\section{A new lower bound}
Our new result is another lower bound on $x_{i}$, which is often, but not always, better than Theorem \ref{thm:mieghem}:
\begin{thm}\label{thm:main}
Let $G$ be a connected graph with principal eigenvector $x$. Then for every $1 \leq i \leq n$:
$$
x_{i} \geq \frac{1}{\sqrt{1+\frac{d_{i}}{(\rho-\rho_{i})^{2}}}}.
$$
\end{thm}

For the proof we need a lemma:
\begin{lem}\cite[p. 148]{TaoVu11}\label{lem:taovu}
Let the Hermitian matrix $A$ be partitioned as
\begin{equation}\label{eq:part}
A=\left[
        \begin{array}{cc}
           a& b^{T}\\
           b& B
        \end{array}
    \right]
\end{equation}
and let $x$ be a unit eigenvector of $A$ corresponding to the eigenvalue $\lambda$. If $\lambda$ is not an eigenvalue of $B$, then
$$
|x_{1}|^{2}=\frac{1}{1+||(\lambda I-B)^{-1}b||^{2}}.
$$
\end{lem}

\begin{proof}[Proof of Theorem \ref{thm:main}]
Without loss of generality, let $i=1$ and suppose that $A$ is partitioned as in \eqref{eq:part}. Then $B$ is the adjacency matrix of $G_{(1)}$. 
As observed before: $\rho>\rho_{1}$. This means that $\rho$ is not an eigenvalue of $B$ and the hypothesis of Lemma \ref{lem:taovu} is satisfied. Thus we have
$$
|x_{1}|^{2}=\frac{1}{1+||(\rho I-B)^{-1}b||^{2}} \geq \frac{1}{1+||(\rho I-B)^{-1}||^{2}||b||^{2}},
$$    
where $||(\rho I-B)^{-1}||$ is the $2$-norm, which is known to be equal to $$\lambda_{\max}((\rho I-B)^{-1})=\frac{1}{\lambda_{\min}(\rho I-B)}=\frac{1}{\rho-\lambda_{\max}(B)}=\frac{1}{\rho-\rho_{1}}.$$
Thus, since $||b||^{2}=d_{1}$ we obtain
$$
|x_{1}|^{2} \geq \frac{1}{1+\frac{d_{1}}{(\rho-\rho_{1})^{2}}}.
$$
\end{proof}

\section{An example}
Consider the following graph:
\begin{figure}[h]
\begin{center}
\includegraphics[width=4in]{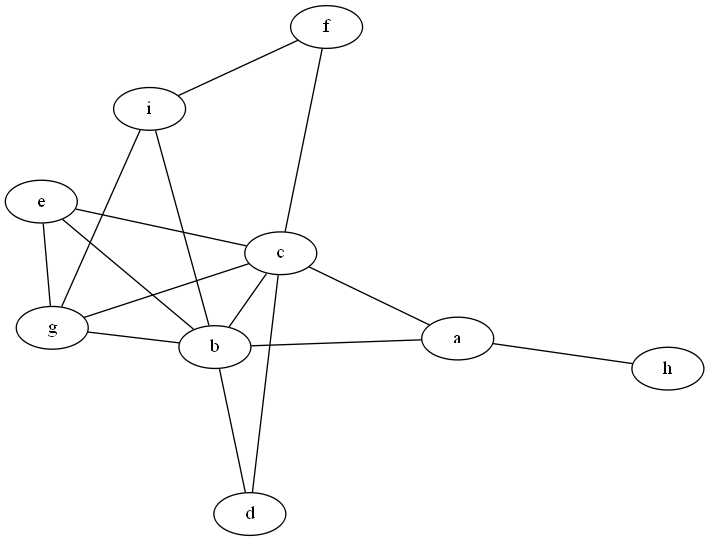} 
\end{center}
\end{figure}

In the table we list the actual values of the principal eigenvector $x$ and the bounds given by all three theorems discussed.

\medskip

\begin{tabular}{|c|l|l|l|l|l|}

\hline

vertex name\ & vertex degree\ & Theorem \ref{thm:mieghem}\ & Theorem \ref{thm:main}\ & $x_{i}$  & Theorem \ref{thm:cg}\ \\ [0.5ex] 

\hline

b& 6&0.39725&0.45901&0.49917&0.5213\\
c&
 6&0.374&0.41636&0.48264&0.5213\\

g& 
4&0.29584&0.33114&0.39818&0.44634\\

a& 

3&0.18076&0.14959&0.26109&0.39654\\

e&

3&0.25233&0.28276&0.34415&0.39654\\

i&3&0.18904&0.16325&0.27064&0.39654\\

d&2&0.17415&0.16949&0.24485&0.33261\\

f&2&0.13045&0.096049&0.18786&0.33261\\

h&1&0.044799&0.016093&0.065114&0.24198\\

\hline
\end{tabular}

\medskip 

As the table makes clear, Theorems \ref{thm:main} and Theorem \ref{thm:mieghem} are, generally speaking, incomparable. Nevertheless, a rule of thumb may be discerned as to when is one better than the other: Theorem \ref{thm:main} works better for vertices of higher degree and Theorem \ref{thm:mieghem} for vertices of low degree. As vertices $a,e,i$ show, however, this rule of thumb is not perfect.

\bibliographystyle{abbrv}
\bibliography{nuim}
\end{document}